\documentclass[12pt,english,reqno]{amsart}
\usepackage{geometry}
\usepackage{amsmath,amssymb,amsthm,bbm,mathtools,comment}
\usepackage{amsfonts}
\usepackage[shortlabels]{enumitem}
\usepackage[pdftex,colorlinks,backref=page,citecolor=blue]{hyperref}
\usepackage[mathscr]{euscript}
\usepackage[usenames,dvipsnames]{color}
\usepackage{tikz,babel,adjustbox}
\usepackage[numbers]{natbib}
\usepackage{graphicx}
\usepackage{caption}
\usepackage{subcaption}
\usepackage{verbatim}
\usepackage{array}
\usepackage[frame,cmtip,arrow,matrix,line,graph,curve]{xy}
\usepackage{graphpap, pstricks}
\usepackage{etoolbox}
\usepackage{pifont}
\usepackage[final]{microtype}
\usepackage{cmtiup}

\hypersetup{pdfpagemode=UseNone,pdfstartview={XYZ null null 1.00}}
\usetikzlibrary{shapes.misc,calc,intersections,patterns,decorations.pathreplacing}
\usetikzlibrary{arrows,arrows.meta,shapes,positioning,decorations.markings}
\tikzstyle arrowstyle=[scale=1]

\allowdisplaybreaks

\makeatletter
\def\@settitle{\begin{center}%
		\bfseries\Large
		\@title
	\end{center}%
}
\patchcmd{\@setauthors}{\MakeUppercase}{\normalsize}{}{}
\makeatother

\theoremstyle{plain}
\newtheorem{theorem}{Theorem}[section]		
\newtheorem{lemma}[theorem]{Lemma}
\newtheorem{claim}[theorem]{Claim}
\newtheorem{proposition}[theorem]{Proposition}

\newtheorem{problem}[theorem]{Problem}

\theoremstyle{remark}

\newcommand{\beq}[1]{\begin{equation}\label{#1}}
\newcommand{\enq}[0]{\end{equation}}
\def\Prob{\mathbb{P}}
\def\EE{\mathbb{E}}
\def\NN{\mathbb{N}}
\def\CC{\mathscr{C}}
\def\DD{\mathfrak{D}}

\providecommand{\abs}[1]{\left\lvert#1\right\rvert}

\makeatletter
\def\imod#1{\allowbreak\mkern10mu({\operator@font mod}\,\,#1)}
\makeatother

\begin{document}

\title{Colouring random subgraphs}

\author{Boris Bukh}
\address{Department of Mathematics, Carnegie Mellon University, Pittsburgh, PA 15213, USA}
\email{bbukh@math.cmu.edu}

\author{Michael Krivelevich}
\address{School of Mathematical Sciences, Tel Aviv University, Tel Aviv 6997801, Israel}
\email{krivelev@tauex.tau.ac.il}

\author{Bhargav Narayanan}
\address{Department of Mathematics, Rutgers University, Piscataway, NJ 08854, USA}
\email{narayanan@math.rutgers.edu}

\date{6 January 2025}

\begin{abstract}
	We study several basic problems about colouring the $p$-random subgraph $G_p$ of an arbitrary graph $G$, focusing primarily on the chromatic number and colouring number of $G_p$. In particular, we show that there exist infinitely many $k$-regular graphs $G$ for which the colouring number (i.e., degeneracy) of $G_{1/2}$ is at most $k/3 + o(k)$ with high probability, thus disproving the natural prediction that such random graphs must have colouring number at least $k/2 - o(k)$.
\end{abstract}
\maketitle

\section{Introduction}
For a graph $G$ and $p \in (0,1)$, let $G_p$ denote the random subgraph of $G$ obtained by randomly including each edge of $G$ independently  with probability $p$. Here, we shall study some basic questions about properly colouring the vertices of $G_p$ for fixed $p \in (0,1)$. Studying typical coloring properties of a random subgraph of a graph with given parameters is a very natural setup in the wide context of random graphs. We are also motivated partly by the following old question of Erd\H{o}s and Hajnal~\citep{EH1,EH2} that remains frustratingly open: is it true that for every pair $t,g \in \NN$, there exists a $k = k(t,g) \in \NN$ such that any graph with chromatic number at least $k$ contains a subgraph with chromatic number at least $t$ and girth at least $g$? A natural step towards this question of Erd\H{o}s and Hajnal --- motivated by Erd\H{o}s' randomised construction (see~\citep{textbook}) of graphs of large girth and chromatic number --- is to study the colouring-related properties of the random subgraph $G_p$ of an arbitrary graph $G$ of large chromatic number.

Concretely, we shall focus on the two problems that we next describe. First, we study the following `chromatic number problem': for $k \to \infty$, given an arbitrary graph $G$ with chromatic number $\chi(G) = k$, what can we say (asymptotically) about the chromatic number $\chi(G_p)$ of the random graph $G_{p}$? Second, as a more approachable weakening of the chromatic number problem where we restrict our attention to \emph{greedy} colourings, we also study the following `colouring number problem': for $k \to \infty$, given an arbitrary graph $G$ with minimum degree $\delta(G) = k$, what can we say (again, asymptotically) about the colouring number $\CC(G_p)$ of the random graph $G_{p}$? Here and later, the colouring number $\CC(G)$ of a graph $G$ is the minimum integer $k$ such that every subgraph $G'$ of $G$ has a vertex of degree less than $k$. Equivalently, the $t$-core of $G$ is the maximal subgraph of $G$ in which all vertices have degree at least $t$; the coloring number is the largest $k$ such that the $(k-1)$-core is non-empty.

Towards the first of our two primary questions, the main problem --- popularised by the first author, but certainly natural enough to have been independently considered by other researchers --- is the following.
\begin{problem}\label{prob:chrom}
As $k \to \infty$, is it true that for any graph $G$ with $\chi(G) = k$, we have  \[\EE [\chi(G_{1/2})] =\Omega( k/\log k)?\]
\end{problem}

The lower bound of $k / \log k$ in Problem \ref{prob:chrom} is natural, and best possible if true; indeed, for the complete graph $G=K_{k}$, the classical result of Bollob\'as~\citep{BB} pinning down the chromatic number of dense Erd\H{o}s--R\'enyi random graphs asserts that $\chi(G_{1/2}) \sim k/(2\log_2 k)$ with high probability.

Problem~\ref{prob:chrom} strikes us as a rather basic question; however, not much appears to be known, and the state of the art is as follows. First, for any graph $G$ with $\chi(G) = k$, since $G_{1/2}$ and its complement (in $G$) have the same distribution, it follows from a simple product-colouring argument that $\EE [\chi(G_{1/2})] \ge k^{1/2}$; a similar argument (using a random partition into $r$ parts) shows that $\EE [\chi(G_{1/r})] \ge k^{1/r}$ for any $r \in \NN$. The argument in~\citep{AKS} shows that $\chi(G_{1/2}) = \Omega(k / \log n)$ holds with high probability, where $n$ is the number of vertices of $G$. Finally, Mohar and Wu~\citep{Mohar} have settled the fractional analogue of Problem \ref{prob:chrom} in the affirmative. Specifically, it was proven in \citep{Mohar} that if $G$ has fractional chromatic number $k$, then with probability $1-o_k(1)$ the fractional chromatic number of the random subgraph $G_{1/2}$ is at least $k/8\log_2(4k)$.

Our primary contribution towards the chromatic number problem is an extension of the work of Shinkar~\citep{Shinkar} studying `large deviations' of $\chi(G_{1/2})$. Taken together, our results give bounds for the entire lower tail of $\chi(G_{1/2})$; the first and the third bounds in the result below are due to Shinkar, while our contribution here is a proof of the second bound.
\begin{theorem}\label{t:chrom}
	For any graph $G$ with $\chi(G) = k$, we have
	\[
		\Prob\left(\chi(G_{1/2}) \le d\right)\le
		\begin{cases}
			\exp (-\Omega((k^{1/2}-d)^2 / k^{1/2}) & \text{for } k^{1/2}/2 \le d \le k^{1/2},             \\
			\exp (-\Omega(k/d))                    & \text{for } k^{1/3} \le d \le k^{1/2}/2, \text{ and} \\
			\exp (-\Omega(k(k-d^3))/d^3)           & \text{for } d \le k^{1/3}.
		\end{cases}
	\]
\end{theorem}

Towards the second of our two primary questions, we raise the following problem.

\begin{problem}\label{p:degen}
As $k \to \infty$, determine the largest $\DD(k)$ for which we have \[\Prob\left(\CC(G_{1/2}) \ge \DD(k)\right) \ge 1/2\]
for all graphs $G$ with $\delta(G) \ge k$.
\end{problem}
In other words, Problem~\ref{p:degen} asks the following: as $k \to \infty$, what is the best possible lower bound on (the probable value of) the colouring number $\CC(G_{1/2})$ that holds for all graphs $G$ with minimum degree $\delta(G)\ge k$? Problem~\ref{p:degen} is the `degree-analogue' of Problem~\ref{prob:chrom}, replacing proper colourings with (the more tractable) proper \emph{greedy} colourings, and the chromatic number with the minimum degree. This is motivated in large part by the degree-analogue --- due to Thomassen~\citep{Thomassen} and also wide open --- of the aforementioned problem of Erd\H{o}s--Hajnal~\citep{EH1,EH2}: is it true that for every pair $t,g \in \NN$, there exists a $k = k(t,g) \in \NN$ such that any graph with average degree at least $k$ contains a subgraph with average degree at least $t$ and girth at least $g$?

Let us point out that Problem~\ref{p:degen} and its variants arise naturally in some other (non-mathematical) contexts as well. First, several variants of Problem~\ref{p:degen} have been studied by biologists, sociologists and theoretical computer scientists as models of `cascading failures' in networks; see~\citep{Blume,Watts} and the references therein, for example. Second, we note that Problem~\ref{p:degen} can also be recast in the language of bootstrap percolation~\citep{CLR}. Bootstrap percolation on a graph $G$ is a model --- originating in statistical physics --- for the spread of infection on $G$ defined as follows: starting with an initially infected set of vertices $A$, infection spreads along the edges of $G$, where a vertex of $G$ gets infected if the number of its (previously) infected neighbours in $G$ exceeds a specified threshold, and $A$ is said to percolate if all the vertices of $G$ are eventually infected. There is by now a large body of (mathematical) work devoted to understanding the percolating sets for various graph families (see~\citep{BP, JRTV}, for example), and Problem~\ref{p:degen} may also be rephrased in this language: given a graph $G$ with $\delta(G)= k$, we are looking to understand for what $t = t(k)$ we can guarantee that the set $A$ of vertices of degree at most $t$ in $G_{1/2}$ percolates in bootstrap percolation on $G_{1/2}$ with the threshold $(\deg(v, G_{1/2}) - t)$ at each vertex $v$ of $G_{1/2}$ (or in other words, for what $t=t(k)$ we can guarantee that the $t$-core of $G_{1/2}$ is non-empty).

It is clear from considering the complete graph $G = K_{k+1}$ that we have the upper bound $\DD(k) \le k/2 + o(k)$. On the other hand, if an $n$-vertex graph $G$ satisfies $\delta(G)=k$, then $G$ has
$\geq kn/2$ edges, and so Chernoff bound \cite[Theorem~A.1.1]{textbook} implies that $\Pr[e(G_{1/2})<kn/4-\sqrt{kn/2}]<e^{-2}<1/2$; using this
and the well-known fact that any graph of average degree $d$ contains a subgraph of minimum degree at least $d/2$, it follows that $\DD(k) \ge k/4 - o(k)$.

While the upper bound of $\DD(k) \le k/2 + o(k)$ seems like the natural guess for the truth, the following result --- our main contribution towards the colouring number problem, and our most significant result here --- shows that this is not the case.

\begin{theorem}\label{t:degen-main}
	As $k \to \infty$, we have $\DD(k) \le k/3 + o(k)$.
\end{theorem}

In more detail, the proof of Theorem \ref{t:degen-main} shows (for all $k \in \NN$ divisible by 3) that there exist large $k$-regular graphs $G$ for which the $t$-core of $G_{1/2}$ is empty with high probability for some $t = k/3 + o(k)$. Our next result, stated below, serves to illustrate some of the subtleties that arise in studying Problem~\ref{p:degen}.

\begin{theorem}\label{t:degen-sec}
	For every $\alpha>0$, there exists a $\beta>0$ such that for infinitely many $k\in\NN$, there exist arbitrarily large $k$-regular graphs $G$ for which the following holds with high probability (as $k \to \infty$): any induced subgraph $H$ of $G_{1/2}$ with $\delta(H) \ge k/4+\alpha k$ satisfies $|V(H)|/|V(G)| = O((1-\beta)^{k^2})$.
\end{theorem}

In other words, Theorem \ref{t:degen-sec} asserts (for infinitely many $k\in\NN$) that there exist large $k$-regular graphs $G$ for which the $t$-core of $G_{1/2}$ is \emph{just barely} non-empty (i.e., is  very small relative to $G$) for any $t = k/4 + o(k)$. In the light of this, it seems clear to us that improving on the easy lower bound of $\DD(k) \ge k/4 + o(k)$ sketched above is likely to require some interesting ideas.

This paper is organised as follows. After covering some preliminaries in Section \ref{s:prelim}, we give the proof of Theorem \ref{t:chrom} in Section \ref{s:chrom} and the proofs of Theorem \ref{t:degen-main} and Theorem \ref{t:degen-sec} in Section \ref{s:col}. Finally, we conclude in Section \ref{s:conc} with a discussion of open problems and directions for further work.

\section{Preliminaries}\label{s:prelim}
We start by establishing some notation and collecting together some tools that we will rely on in the sequel.

Our graph theoretic notation is for the most part standard; we refer the reader to~\citep{BBbook} for terms not defined here. That said, we remind the reader of a few standard notions that come up frequently in this paper.

First, recall that the \emph{chromatic number $\chi(G)$} of a graph $G$ is the smallest number of colours needed to properly colour the vertices of $G$, i.e., to colour the vertices in such a way that no two adjacent vertices share the same colour.

Next, following Erd\H{o}s and Hajnal~\citep{EH3}, the \emph{colouring number $\CC(G)$} of a graph $G$ is the least number $c$ for which there exists an ordering of the vertices of $G$ in which each vertex has fewer than $c$ neighbours preceding it in the ordering; this parameter --- also (essentially) called the degeneracy or the core number --- is the number of colours used by the natural greedy algorithm for properly colouring the vertices of $G$.

We shall also need two notions of graph boundaries: for a subset $S\subset V(G)$ of the vertices of a graph $G$, its \emph{vertex boundary $\partial S$} consists of those vertices of $G$ not in $S$ that are adjacent to at least one vertex in $S$, and its \emph{edge boundary $\nabla S$} consists of those edges of $G$ with one end in $S$ and the other in $\overline S$. Overloading this notation slightly, for a subset $S\subset V(G)$ of the vertices of a directed graph $G$, its vertex boundary $\partial S$ consists of those out-neighbours of $S$ that are not in $S$, and similarly, its edge boundary $\nabla S$ consists of those edges of $G$ directed from $S$ to $\overline S$.

We need a standard bound for the number of connected components in a graph of given maximum degree; it can be found, e.g., in \citep{BFM} (see Lemma 2 there), along with a proof.
\begin{lemma}\label{lem:treelemma}
	For a graph $G$ of maximum degree $\Delta$, the number of connected, $t$-edge subgraphs of $G$ containing a given vertex is less than $(e\Delta)^t$. 
\end{lemma}

Another fairly standard fact we utilize is a quantitative connection between eigenvalues and edge distribution in regular graphs. For a graph $G$ its eigenvalues are those of its adjacency matrix $A(G)$. The following bound is due to Alon and Milman~\citep{AM}.
\begin{lemma}\label{lem:AM}
	Let $G$ be a $d$-regular graph on $n$ vertices with the second largest eigenvalue $\lambda$. Then for every subset $S\subset V$, one has
	\[|\nabla S|\ge \frac{(d-\lambda)|S|(n-|S|)}{n}.\] 
\end{lemma}

\section{Chromatic number}\label{s:chrom}
First, following~\citep{AKS}, we record (in slightly greater generality) a proof of the fact that for any $n$-vertex graph $G$ with $\chi(G) = k$,  we have $\chi(G_{1/2}) = \Omega(k / \log n)$ with high probability (as $n \to \infty$).
\begin{proposition}
	For any $n$-vertex graph $G$ with $\chi(G) = k$ and any $0<p<1$, we have
	\[\chi(G_p)\ge \frac{pk}{2\log n}\] with high probability (as $n\to\infty$).
\end{proposition}
\begin{proof}
	The probability that there exists a set $V'\subset V(G)$ for which
	\begin{enumerate}[(a)]
		\item the induced subgraph $G[V']$ has minimum degree at least ${2\log n}{/p}$, and
		\item $V'$ becomes an independent set in $G_p$,
	\end{enumerate}
	is at most
	\[
		\sum_{m={2\log n}/{p}}^n \binom{n}{m}(1-p)^{(m \log n)/p}\leq \sum_{m={2\log n}/{p}}^n \left(\frac{en}{m}\cdot \frac{1}{n}\right)^m=o(1).
	\]
	As every $t$-chromatic graph contains a subgraph of minimum degree at least $t-1$ (a color-critical subgraph), the above implies that with high probability, any independent set $V'$ in $G_p$ induces a subgraph of $G_p$ of chromatic number at most $(2\log n)/p$ in~$G$.
	As $G_p$ can be partitioned into $\chi(G_p)$ independent sets (by definition), the result follows.
\end{proof}

We now prove Theorem \ref{t:chrom}.
\begin{proof}[Proof of Theorem \ref{t:chrom}]
	As mentioned, our contribution is the second bound in the statement of the theorem. Let $G=(V,E)$ be a graph with $\chi(G)=k$.  Our goal is to estimate from above the probability $\Prob(\chi(G_{1/2}) \le d)$.

	Fix an optimal coloring $V=V_1\cup\ldots\cup V_k$ of $G$. Equipartition $[k]=I_1\cup\ldots\cup I_s$ with $|I_j|\ge 2d^2$ and $s=\Theta(k/d^2)$. Set $G_j=G[\cup_{i\in I_j}V_i]$ for $1\le j\le s$, and note that $\chi(G_j)=|I_j|\ge 2d^2$. If $\chi(G_{1/2})\le d$, then the chromatic numbers of all the random subgraphs $(G_j)_{1/2}$ are at most $d$. Observe crucially that these events are independent as the graphs $G_j$ do not share any vertices, and thus edges. Hence
	\[
		\Prob(\chi(G_{1/2}) \le d)\le  \prod_{i=1}^s \Prob(\chi((G_i)_{1/2})\le d)\,.
	\]
	Recall that $\chi(G_j)\ge 2d^2$. Hence, as explained in the introduction, $\EE [\chi((G_j)_{1/2})] \ge \sqrt{2}d$. Using the first bound in the statement of the theorem (proved using Doob martingales as outlined by Shinkar~\citep{Shinkar}), we get $\Prob(\chi((G_i)_{1/2})\le d)\le e^{-cd}$ for some absolute constant $c>0$. It follows that
	\[
		\Prob(\chi(G_{1/2}) \le d)\le \left(e^{-cd}\right)^s=\exp(-\Theta(k/d))\,,
	\]
	as required.
\end{proof}

A twist on the above idea also provides a new and fairly simple proof of the third statement of Theorem \ref{t:chrom} for the range $d=O( k^{1/4})$. Here is an outline. For $G$ with an optimal coloring $V=V_1\cup\ldots\cup V_k$, fix a collection of subsets $I_1,\ldots,I_s\subset [k]$ with $|I_j|\ge 2d^2$ and $s=\Theta(k^2/d^4)$ so that $|I_i\cap I_j|\le 1$ for every $1\le i\ne j\le s$; the existence of such a collection is a fairly standard fact in design theory. Let $G_j=G[\cup_{i\in I_j}V_i]$, $1\le j\le s$. The events
$A_j=\{\chi((G_j)_{1/2})\le d\}$ are again independent, and each happens with probability at most $e^{-cd}$. It follows that $\Prob(\chi(G_{1/2}) \le d)\le  (e^{-cd})^s= \exp(-\Theta(k^2/d^3))$.

\section{Colouring number}\label{s:col}
Our proofs of Theorem \ref{t:degen-main} and Theorem \ref{t:degen-sec} rely on the existence of good expander graphs. Here, we make use of specific graphs that happen to be Ramanujan, but any family of sufficiently strong expanders should suffice.

We start with the proof of Theorem \ref{t:degen-main}.

\begin{proof}[Proof of Theorem \ref{t:degen-main}]
	Since $\DD(k)$ is a non-decreasing function of $k$, it suffices to only consider $k$ that are divisible by $3$. Given any small $0 < \alpha < 1/100$ and $k \in \NN$ with $3 \mid k$, we shall construct a $k$-regular graph $G$ (infinitely many, in fact) with the property that, for $t=k/3 + \alpha k$, the $t$-core of $G_{1/2}$ is empty with probability $1-o(1)$ as $k \to \infty$; clearly, this suffices to prove the result.

	We need the following well-known fact: there are positive constants $c',c''$ such that  for infinitely many $n$ there is a 3-regular graph $H$ on $n$ vertices without cycles shorter than $c'\log n$, and with all eigenvalues $\lambda_i$ but the first one $\lambda_1=3$ satisfying $\lambda_i\le 3-c''$. See~\citep{Chiu}, say, for a proof of this fact.

	Choose $H$ as above, and let $G$ be a $(k/3)$-blow-up of $H$, i.e., $G$ is obtained from $H$ by replacing each vertex of $H$ by an independent set of size $k/3$ --- we call these sets (and interchangeably, the vertices of $H$) \emph{super-vertices} ---  and by replacing each edge of $H$ by a complete bipartite graph in $G$ between the corresponding super-vertices.

	In order to help the reader to grasp our argument, let us state that it implements and analyses the following bootstrap percolation-type process on $H$. First, we form a random subset $R$ of \emph{protected edges} of $H$, where an edge $e\in E(H)$ is declared protected independently and with probability $p=\exp\{-\Theta(k)\}$; protected edges correspond to complete bipartite graphs between the super-vertices of $H$ in which in the random subgraph $G_{1/2}$ there is a vertex of degree at least $k/6+\alpha k/2$; clearly the events corresponding to the edges of $H$ becoming protected are independent for different edges of $H$, and happen each with probability exponentially small in $k$. Then a random vertex $r$ of $H$ is chosen; in the argument this will be a super-vertex of $H$ all of whose incident edges get erased in the first round of deletions. Now, consider the following propagation process. We start with $V_0=\{r\}$, and at each step update $V_0$ by adding to it all the vertices of $H$ outside of $V_0$ that have at least two neighbours in $V_0$, or alternatively have at least one neighbor in $V_0$ and are not incident to any protected edge from $R$. We will prove that if $H$ is a good expander with logarithmic girth, then typically, the above propagation process ends with $V_0$ consuming all the vertices of $H$; this corresponds to the random subgraph $G_{1/2}$ having an empty $(k/3+\alpha k)$-core, as desired.

	We say that a vertex of $G$ \emph{survives} or \emph{lives} if it is present in the $t$-core of $G_{1/2}$, and that it \emph{dies} otherwise; similarly, we say that a super-vertex of $H$ \emph{dies} if none of its constituent vertices survive, and that it \emph{lives} or \emph{survives} otherwise.

	We shall, for technical reasons, construct $G_{1/2}$ by deleting the edges of $G$ in two rounds: in the first round, each edge of $G$ is independently sampled with probability $\alpha/3$, in the second round, each edge of $G$ is independently sampled with probability $(1/2 - \alpha/3) / (1- \alpha/3) \ge 1/2 - \alpha/2$, and finally, all the sampled edges are deleted to form $G_{1/2}$.

	Since $n \ge (\alpha/3)^{-k^3}$, there is, with high probability \emph{over the random deletions in the first round}, some super-vertex for which all the edges incident to it in $G$ are deleted in the first round. Therefore, let us condition on the event that all the edges incident to some super-vertex are deleted in the first round; let $r$ be any such super-vertex. Clearly, such an $r$ dies. Let $T_r$ be the connected set of dead super-vertices containing $r$; we claim that $|T_r| = n$ with high probability \emph{over the random deletions in the second round}. Since the two rounds of deletions are independent, this claim clearly implies that the $t$-core of $G_{1/2}$ is empty with high probability.

	The rest of the proof is devoted to the proof of the claim above, namely that for any fixed super-vertex $r$, conditional on $r$ dying after the first round of deletions, the second round of deletions guarantee that $|T_r| = n$ with high probability. In what follows, we fix an arbitrary super-vertex $r$, abbreviate $T_r$ by $T$, and write $\Prob_r$ for the probability over the random deletions in the second round, conditioned on $r$ dying in the first round.

	We now need slightly different arguments based on how large $m = |T|$ might be. Before we turn to this, we observe that
	\begin{enumerate}
		\item if $m < n$, then since $H$ is connected, the vertex boundary $\partial T$ of $T$ in $H$ is both non-empty and necessarily contained in the set of surviving super-vertices, and
		\item for each surviving super-vertex $v\in\partial T$, there is at least one vertex $v^* \in V(G)$ contained in $v$ that survives.
	\end{enumerate}

	First, we handle the case where $1 \le m < 99n/100$ by a union bound over the potential choices of $T$. Our task then is to bound, for all choices of $T_0$ with $|T_0| = m$, the probability $\Prob_r(T=T_0)$.

	Consider any connected set $T_0$ of $m$ super-vertices containing $r$. Due to our choice of $H$, Lemma \ref{lem:AM}, and since  $|T_0| = m < 99n/100$, we know that $|\partial T_0|>c_1 m$ (for some universal $c_1 > 0$). Let $v_1,v_2,\dots,v_\ell$ be a maximal independent set of surviving super-vertices in $H[\partial T_0]$, and note that since each vertex in $H[\partial T_0]$ has degree at most 2, we must have $\ell \ge |\partial T_0|/3 \ge c_1 m/3$. Next, note that if $T=T_0$, then there must exist vertices $v^*_1 \in v_1, v^*_2 \in v_2,\dots, v^*_\ell \in v_\ell$ of $G$ that also survive.

	For any such choice of vertices $v^*_1, v^*_2,\dots, v^*_\ell$, we shall now estimate the probability, over the second round of deletions, that these vertices survive. By virtue of how $G$ is constructed from $H$, it is clear that $v^*_i$ is not adjacent to any of the vertices $v^*_1,v^*_2,\dots,v^*_{i-1}$ for all $1\le i \le \ell$. This allows us to bound
	\begin{align}\label{earlier-survive}
		\Prob_r \left(  v^*_1,\dotsc, v^*_{\ell} \text{ survive}  \mid T_0 \textrm{ dies}  \right)
		 & \le \prod_{i=1}^{\ell} \Prob_r
		\begin{pmatrix}
			\,\text{at least }k/3+\alpha k\text{ edges from }v^*_i\,\, \\
			\text{to }\overline{T_0} \text{ survive the second round}
		\end{pmatrix}
		\notag                                                                                                           \\
		 & \leq \prod_{i=1}^{\ell} \Prob\left(\operatorname{Binom}(2k/3,1/2 + \alpha/2) \ge k/3 + \alpha k\right) \notag \\
		 & \leq (1+c_2)^{-k\ell},
	\end{align}
	where $c_2 > 0$ is a constant depending on $\alpha$ alone.

	Using the fact that $\ell \ge c_1 m/3$, a union bound over all potential choices of $v^*_1,v^*_2,\dots,v^*_{\ell}$ --- of which there are at most $(k/3)^{2m}$ since $|\partial T_0| \le 2|T_0| = 2m$ --- yields the estimate
	\[
		\Prob_r\left(T=T_0 \right) \leq (k/3)^{2m} (1+c_2)^{- c_1 k m / 3}.
	\]
	Finally, the number of connected sets $T_0$ of size $m$ that contain the fixed root $r$ is, by Lemma \ref{lem:treelemma}, at most $(3e)^m$. Thus, it follows again from the union bound that
	\beq{smallT}
	\Prob_r\left( 1 \le |T| < 99 n/100 \right) \le \sum_{m=1}^{99 n/100} \left(e k^2 / 3\right)^m \left(1+c_2\right)^{- c_1 k m / 3} = o(1),
	\enq
	with the last asymptotic estimate holding in the limit of $k \to \infty$.

	Next, we deal with the possibility that $99n/100 \le m < n$. In this case, note that (by the definition of $T$), every super-vertex $v\in \overline{T}$ sends at most one edge to $T$, and hence has at least two neighbors in $\overline{T}$. Let $S$ be a connected component in $H[\overline{T}]$ and put $s = |S|$; since $S$ has minimum degree $2$, it contains a cycle, and since $H$  has girth at least $c' \log n$ (for some universal $c'>0$), this implies that $0.01n \ge s = |S| \ge c'\log n$. Observe that since $S$ is a connected component of $H[\overline{T}]$, it must be the case that $\partial S \subset T$. As $|\nabla S| \ge c'' s$ for an absolute constant $c_3>0$, again due to our choice of $H$ and Lemma \ref{lem:AM}. Since  $|S| = s \le 0.01n$, and since each super-vertex of $S$ has at most one neighbour outside $S$, we conclude that at least $c'' s$ super-vertices in $S$ have a neighbour in $T$, so $|\partial T \cap S| \ge c'' s$. As before, we may find a set $I \subset \partial T \cap S$ of $c'' s /3$ super-vertices that are independent in $H[\partial T]$. As we argued for~\eqref{earlier-survive}, the probability of the super-vertices in $I$ all surviving conditional on $T$ dying is at most
	\[\left((k/3) (1+c_2)^{-k}\right)^{c'' s /3},\]
	where $c_2 > 0$ is, exactly as before, a constant depending on $\alpha$ alone. Then, again invoking Lemma \ref{lem:treelemma}, by a union bound over the choice of a connected $S$ in $H$ of size $s$, and $I \subset S$ (which can be chosen in at most $2^s$ ways), we get
	\beq{largeT}
	\Prob_r\left( 99 n/100 < |T| < n \right) \le \sum_{s=c' \log n}^{n/100} n (3e)^s 2^s \left((k/3) (1+c_2)^{-k}\right)^{c'' s /3} = o(1),
	\enq
	with the last asymptotic estimate holding in the limit of $k \to \infty$.
	The desired claim, namely that $\Prob_r(|T| < n) = o(1)$, follows from~\eqref{smallT} and~\eqref{largeT}, and the proof is complete.
\end{proof}

The bottleneck in the proof of Theorem \ref{t:degen-main} that we just saw comes from the tension between graph expansion and
the impact of having dead neighbours. Specifically, instead of starting
with $3$-regular graphs and looking at the $(k/3+\alpha k)$-core, if we started with $\ell$-regular graphs and looked at the $t$-core,
then having a dead super-neighbour would be a serious mortality risk only if $t>\tfrac{1}{2}(1-1/\ell)k$.
Improving the argument would require using $2$-regular expanders, which clearly do not exist.

To prove Theorem \ref{t:degen-sec} we turn to \emph{directed} expander graphs instead. There do exist directed expander graphs all whose in-degrees
are equal to $2$. A downside to this approach is that the gadgets we now use to form the super-vertices are more complex than mere independent sets.
Consequently, these gadgets contain high-density subgraphs which have a non-neglible chance of surviving in the last
phase of the deletion process when $99n/100\leq m<n$; this explains why a tiny number of vertices survive in Theorem \ref{t:degen-sec}.

We call a directed graph \emph{$d$-regular} if the in-degree and out-degree of each vertex are~$d$. The following lemma follows from a standard probabilistic construction.

\begin{lemma}\label{l:two-reg-expander}
	For all sufficiently large $n$, there exist $2$-regular directed graphs $H$ on $n$ vertices that,	for every non-trivial subset $S \subset V(H)$ of vertices, satisfy
	\[
		\abs{\partial S}\geq c_3 \min(\abs{S},n-\abs{S}),
	\]
	where $c_3>0$ is a universal constant.
\end{lemma}
\begin{proof}
	A uniformly random $2$-regular directed graph on $n$ vertices has this property with high probability as $n \to \infty$. A proof of this fact is, at this point, a routine argument using the configuration model; see~\citep{hoory_linial_wigderson_expanders}, for example, for a similar argument in the context of undirected graphs (that extends to the directed case as well).
\end{proof}

\begin{proof}[Proof of Theorem \ref{t:degen-sec}]
	Given any small $0 < \alpha < 1/16$, fix an integer $2/\alpha \le s \le 4/\alpha$, and let $k$ be any large positive integer that is divisible by $2s$. We shall construct, for infinitely many $n \in \NN$, a $k$-regular graph $G$ on $nk(s+3)(1/2-1/2s)$ vertices with the property that, for
	\[t=k/4+2k/s,
	\] the $t$-core of $G_{1/2}$ has density at most $(1-\delta)^{k^2}$ with probability $1-o(1)$ as $k \to \infty$, where $\delta > 0$ is a constant depending on $\alpha$ alone; clearly, this suffices to prove the result.

	Let $H$ be a $2$-regular directed expander on $n \ge (\alpha/6)^{-k^3}$ vertices as promised by Lemma \ref{l:two-reg-expander}. To describe the blow-up process we use to construct $G$ from $H$, we need to be able to distinguish the in-edges at each vertex of $H$; to that end, two-colour the edges of $H$ (with colours red and blue, say) so that the two in-edges at each vertex are coloured differently. We then build $G$ from $H$ according to the procedure illustrated in \eqref{picpic} as follows; it is routine to verify that this construction indeed produces a $k$-regular graph.
	\begin{enumerate}
		\item Replace each vertex $v$ of $H$ by a disjoint union of $s+3$ independent sets of size $k/2-k/2s$ each;
		      denote these independent sets by $I_1(v),\dots,I_{s+3}(v)$.
		\item For each vertex $v$ of $H$, place a complete bipartite graph between the sets $I_j(v)$ and $I_{j+1}(v)$ for each $1 \le j \le s+2$.
		\item For every \emph{red} directed edge $u\to v$ in $H$ and each $2 \le j \le s+2$, place an arbitrary $(k/2s)$-regular bipartite graph between the sets $I_j(u)$ and $I_1(v)$.
		\item For every \emph{blue} directed edge $u\to v$ in $H$ and each $2 \le j \le s+2$, place an arbitrary $(k/2s)$-regular bipartite graph between the sets $I_j(u)$ and $I_{s+3}(v)$.
	\end{enumerate}
	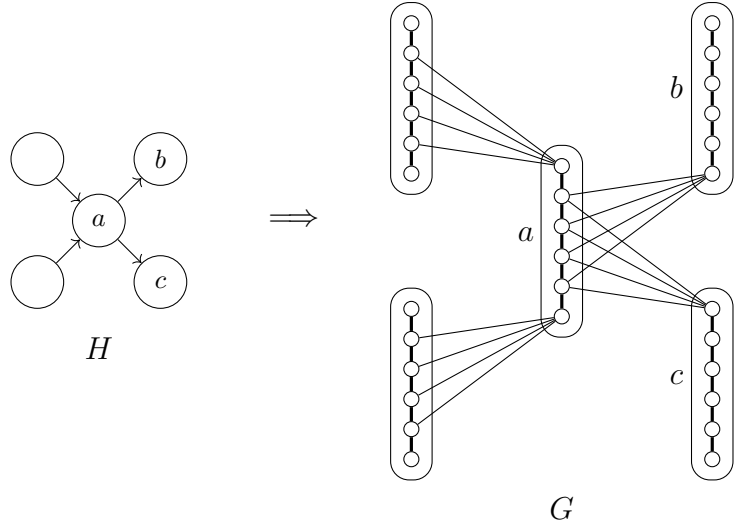
\begin{figure}
		\begin{center}
			\begin{tikzpicture}[baseline=(a.base), before/.style={circle,draw,minimum size=7mm,execute at begin node=\footnotesize}]
				\node[before] (a) {$a$}
				node[before] (b) [above right=8mm] {$b$}
				node[before] (c) [below right=8mm] {$c$}
				node[before] (d) [above left=8mm]  { }
				node[before] (e) [below left=8mm]  { }
				;
				\draw[->] (a) -- (b);
				\draw[->] (a) -- (c);
				\draw[->] (d) -- (a);
				\draw[->] (e) -- (a);
				\draw (0,-1.7) node {$H$};
			\end{tikzpicture}
			\qquad$\implies$\qquad
			\begin{tikzpicture}[baseline=(A3), after/.style={circle,draw,minimum size=2mm,inner sep=0pt}, round/.style={rounded corners=7pt}]
				\def\ht{5}
				\def\yshf{1.9}
				\def\spacing{0.4}
				\foreach \i in {0,1,...,\ht}
					{
						\node[after] (A\i) at ($\spacing*(0,\i)$) {};
						\node[after] (B\i) at ($\spacing*(0,\i)+(2,\yshf)$) {};
						\node[after] (C\i) at ($\spacing*(0,\i)+(2,-\yshf)$) {};
						\node[after] (D\i) at ($\spacing*(0,\i)+(-2,\yshf)$) {};
						\node[after] (E\i) at ($\spacing*(0,\i)+(-2,-\yshf)$) {};
					}
				\foreach \lab/\blob in {a/A,b/B,c/C}
					{
						\node at ($(\blob0.south west)!0.5!(\blob\ht.north west)-(0.4,0)$) {$\lab$};
					}
				\foreach \blob in {A,B,C,D,E}
					{
						\draw[round] ($(\blob0.south west)-(0.2,0.2)$) rectangle ($(\blob\ht.north east)+(0.2,0.2)$);
						\foreach \i [count=\j from 0] in {1,2,...,\ht}
							{
								\draw[very thick] (\blob\i) -- (\blob\j);
							}
					}
				\foreach \i [count=\j from 1] in {2,3,...,\ht}
					{
						\draw (A\j) -- (B0);
						\draw (A\j) -- (C\ht);
						\draw (E\j) -- (A0);
						\draw (D\j) -- (A\ht);
					}
				\draw (0,-2.7) node {$G$};
			\end{tikzpicture}\\[2ex]
		\end{center}
		\caption{From $H$ to $G$: thick edges are complete bipartite graphs of degree $k/2-k/2s$, and thin edges are bipartite graphs of degree $k/2s$.}
		\label{picpic}

	\end{figure}

	To orient the reader, let us say that what follows is an analysis of the following bootstrap percolation-type process on $H$. We form a random subset $R$ of the vertices (namely, those termed `resilient' in the sequel) by placing every vertex $v\in V(H)$ into $R$ independently with probability $\exp\bigl(-\Theta(k^2)\bigr)$. Then, a random initial vertex $r\in V(H)$ is chosen, and at this point, we consider the following propagation process. We start with $V_0=\{r\}$, and at each step, we update $V_0$ by adding to it every out-neighbour $u$  of $V_0$ that satisfies $u \notin R$. The goal is to prove that typically $V_0$ grows to contain all but an exponentially small (in $k^2$) proportion of the vertices of $H$.

	As in the proof of Theorem \ref{t:degen-main}, we delete edges of $G$ in two rounds: we sample the edges in the two rounds independently with probabilities $1/3s \ge \alpha/6$ and $(1/2 - 1/3s) / (1- 1/3s) \ge 1/2 - 1/2s$ respectively, and then delete all the sampled edges. With the same notions of vertices and super-vertices \emph{surviving} and \emph{dying} as in the proof of Theorem \ref{t:degen-main}, we assume that some super-vertex --- we write $r$ for such a super-vertex --- dies in the first round with high probability; this is justified since $n \ge (\alpha/6)^{-k^3}$ is large enough to ensure this. We then write $\Prob_r$ to denote the probability over the random deletions in the second round, conditioned on $r$ dying in the first round.

	Note that the $s+3$ independent sets inside a super-vertex form a path; call a pair of adjacent independent sets $I_j(v)$ and $I_{j+1}(v)$ inside some super-vertex $v$ \emph{resilient} if either the set
	\[
		\{ v^*\in I_j(v) : \text{at least }k/4\text{ edges from }v^*\text{ to }I_{j+1}(v)\text{ survive the second round}\}
	\]
	or the set
	\[
		\{ v^*\in I_{j+1}(v) : \text{at least }k/4\text{ edges from }v^*\text{ to }I_{j}(v)\text{ survive the second round}\}
	\]
	has size at least $k/s$.

	Call a super-vertex $v$ \emph{nearly dead} if each of the sets $I_j(v)$ for $2 \le j \le s+2$ contains fewer than $k/s$ surviving vertices. Note that each dead super-vertex is also nearly dead, and that if $v$ is nearly dead, then the vertices in the sets $I_3(v), I_4(v), \dots, I_{s+1}(v)$ all die since each vertex therein is adjacent to fewer than
	\[ k/s + k/s + k / 2s + k/2s < k/4\]
	surviving vertices (with room to spare).

	Writing $T$ for the set of all nearly dead super-vertices that can be reached along a directed path starting at $r$ in $H$, we observe the following.
	\begin{claim}\label{respair}
		Each $v\in \partial T$ contains a resilient pair.
	\end{claim}
	\begin{proof}
		Let $u\in T$ be a nearly dead in-neighbour of $v$, and consider the directed edge from $u$ to $v$; without loss of generality, suppose that this edge is coloured red. Let $1 \le \ell \le s+3$ be the smallest index for which $I_\ell(v)$ contains at least $k/s$ surviving vertices; by the definition of $v$ not being nearly dead, such an $\ell$ exists and satisfies $\ell \le s+2$.

		Suppose for the sake of contradiction that the sets $I_{\ell}(v)$ and $I_{\ell+1}(v)$ do not form a resilient pair. Then $I_{\ell}(v)$ contains a surviving vertex $v^*$ incident to fewer than $k/4$ edges into $I_{\ell+1}(v)$ that survive the second round of deletions; such a vertex $v^*$ then has to be incident to more than $t-k/4=2k/s$ surviving vertices outside $I_{\ell+1}(v)$. We cannot have $\ell=1$ because the sets $I_3(u), I_4(v), \dots, I_{s+1}(u)$ are all dead, and
		the sets $I_2(u)$ and $I_{s+2}(u)$ contain fewer than $k/s$ surviving vertices each, leaving $v^*$ with fewer than $k/s + k/s = 2k/s$ surviving neighbours outside $I_{\ell+1}(v)$. We cannot have $2\leq \ell\leq s+2$ either because such a $v^*$ is adjacent to at most $k/2s$ vertices inside each of the two out-neighbours of $v$, and fewer than $k/s$ vertices in $I_{\ell-1}(v)$ since this set contains fewer than $k/s$ surviving vertices (by the minimality of $\ell$), leaving $v^*$ with fewer than $k/2s + k/2s + k/s = 2k/s$ surviving neighbours outside $I_{\ell+1}(v)$. Hence the sets $I_{\ell}(v)$ and $I_{\ell+1}(v)$ form a resilient pair, as desired.
	\end{proof}

	Call a super-vertex resilient if it contains a resilient pair, and let $R$ be the set of resilient super-vertices. Since resilience of a super-vertex depends only on the edges of $G$ inside the super-vertex, events of the form $\{v\in R\}$
	are mutually independent for different super-vertices $v$. It is also clear that for each super-vertex $v$, we have
	\begin{align*}
		\Prob_r(v\in R) & \leq (s+2)\binom{k/2-k/2s}{k/s}\Prob(\operatorname{Binom}(k/2-k/2s,1/2+1/2s)\ge k/4 )^{k/s} \\
		                & \leq (1+c_4)^{-k^2}
	\end{align*}
	for some constant $c_4>0$ that depends on $s$ (and thus $\alpha$) alone.

	This gives us a way to estimate the size of $T$: if $\abs{T_0}\leq 99n/100$, then
	\[
		\Prob_r(T=T_0)\leq \left((1+c_4)^{-k^2}\right)^{\abs{\partial T_0}}\leq (1+c_4)^{-c_3 \abs{T_0} k^2 / 100},
	\]
	from which it follows (as in the proof of Theorem \ref{t:degen-main}) that
	\[
		\Prob_r(\abs{T}\leq 99n/100)\leq \sum_{m=1}^{99n/100} (3e)^m (1+c_4)^{-c_3 m k^2 / 100}=o(1)
	\]
	as $k\to\infty$.

	To see that the size of $T$ must be very nearly $n$ (and not just at least $99n/100$), we analyse its complement.  By the definition of $R$ and \eqref{respair}, we have $\partial T\subseteq R$. By Markov's inequality, $|R|<  n(1+c_4/2)^{-k^2}$ with high probability as $k \to \infty$. By our choice of $H$ satisfying Lemma \ref{l:two-reg-expander}, for every $T$ with $n/2\le |T|\le n - n (1+c_4/2)^{-k^2}/c_3$, we have $|\partial T|\ge n(1+c_4/2)^{-k^2}$. This tells us that
	\[\abs{\overline{T}} / n \le (1+c_4/2)^{-k^2}) / c_3 = (1-\beta)^{k^2},\]
	with high probability as $k \to \infty$, where $\beta>0$ is again a constant depending on $\alpha$ alone; this completes the proof.
\end{proof}

\section{Conclusion}\label{s:conc}
A large number of interesting open problems remain; below, we highlight a few that we find particularly appealing.

First, in the context of the chromatic number of problem, all the lower bounds on $\chi(G_{1/r})$ in terms of $\chi(G)$ that we currently have rely crucially on $r$ being an integer. It would be very interesting to prove any lower bound for $\EE[\chi(G_{0.499})]$ that is (asymptotically) better than $\chi(G)^{1/3}$.

Second, in the context of the colouring number of problem, we have been unable to prove any interesting lower bounds for $\DD(k)$. We would not be surprised if the truth is that $\DD(k) = k/3+o(k)$ as $k \to \infty$, but even showing that $\DD(k) \ge k/3.99 +o(k)$ appears to be a challenging problem, as evidenced by Theorem \ref{t:degen-sec}. In fact, we (somewhat embarrassingly) do not know if there exists a function $f(k) \to \infty$ as $k\to\infty$ such that $\DD(k) \ge k/4+f(k)$.

\section*{Acknowledgements}
The first author was supported by NSF grants DMS-2154063 and DMS-1555149, and the third author was supported by NSF grants DMS-2237138 and CCF-1814409, as well as a Sloan Research Fellowship. Additionally, we also wish to thank Yury Person and Benny Sudakov for valuable discussions.

\bibliographystyle{amsplain}
\bibliography{col_rand_subgraphs}

\end{document}